\documentclass{article}
\usepackage{amssymb,amsmath}
\usepackage{graphicx}

\def\qed{\hfill {\hbox{${\vcenter{\vbox{               %HOLLOW SQUARE
   \hrule height 0.4pt\hbox{\vrule width 0.4pt height 6pt
   \kern5pt\vrule width 0.4pt}\hrule height 0.4pt}}}$}}}

\def\hat{\widehat}

\newtheorem{theorem}{Theorem}
\newtheorem{definition}{Definition}

\newtheorem{cor}[theorem]{Corollary}
\newtheorem{example}{Example}
\newtheorem{remark}[example]{Remark}
\newtheorem{conjecture}{Conjecture}

\newenvironment{proof}[1][Proof]{\smallskip\noindent{\bf #1.}\quad}{\qed\par\medskip}

\date{}

\title{\Large \textbf{Isotopy and Homotopy Invariants of Classical and Virtual Pseudoknots}}

\author{
Fran\c{c}ois Dorais\footnote{francois.g.dorais@dartmouth.edu, Dartmouth College, Hanover, NH 03755, United States}\hspace{1cm}
Allison Henrich\footnote{henricha@seattleu.edu, Seattle University, Seattle, WA 98122, United States}\hspace{1cm}
Slavik Jablan\footnote{sjablan@gmail.com, The Mathematical Institute, Belgrade, 11000, Serbia}\hspace{1cm}
Inga Johnson\footnote{ijohnson@willamette.edu, Willamette University, Salem, OR 97301, United States}
}

\begin{document}

\maketitle

\begin{abstract}
Pseudodiagrams are knot or link diagrams where some of the crossing information is missing. Pseudoknots are equivalence classes of pseudodiagrams, where equivalence is generated by a natural set of Reidemeister moves. In this paper, we introduce a Gauss-diagrammatic theory for pseudoknots which gives rise to the notion of a virtual pseudoknot. We provide new, easily computable isotopy and homotopy invariants for classical and virtual pseudodiagrams. We also give tables of unknotting numbers for homotopically trivial pseudoknots and homotopy classes of homotopically nontrivial pseudoknots. Since pseudoknots are closely related to singular knots, this work also has implications for the classification of classical and virtual singular knots.
\end{abstract}
\bigskip



%\tableofcontents

%%%%%%%%%Introduction%%%%%%%%%%
\section{Introduction}
%%%In this section, we define pseudoknots (cite Pseudoknot, Hanaki, and SMALL papers--give brief history), give corresponding theory of Gauss-type diagrams, define virtual pseudoknots; also, introduce the definition of homotopy (cite Pseudoknot paper); describe organization of paper.%%%

Pseudodiagrams are knot or link diagrams where some of the crossing information is missing.  Where there is missing information, instead of a crossing with clearly marked over- and under-strands, a {\em precrossing} or double-point of the curve appears in the diagram.   Pseudodiagrams of spatial graphs, knots and links were first introduced as potential models for biological applications by Hanaki~\cite{hanaki}. Classical and virtual pseudodiagrams were further studied in~\cite{SMALL}.  Henrich et al.~\cite{pseudoknot} first defined pseudoknots as equivalence classes of pseudodiagrams up to rigid vertex isotopy and a collection of natural Reidemeister moves.  This collection of moves includes the classical Reidemeister (R) moves and a number of additional pseudo-Reidemeister (PR) moves as seen in Figure~\ref{Rmoves}.
 
 \begin{figure}[htbp]
\begin{center}
\includegraphics[scale=.75]{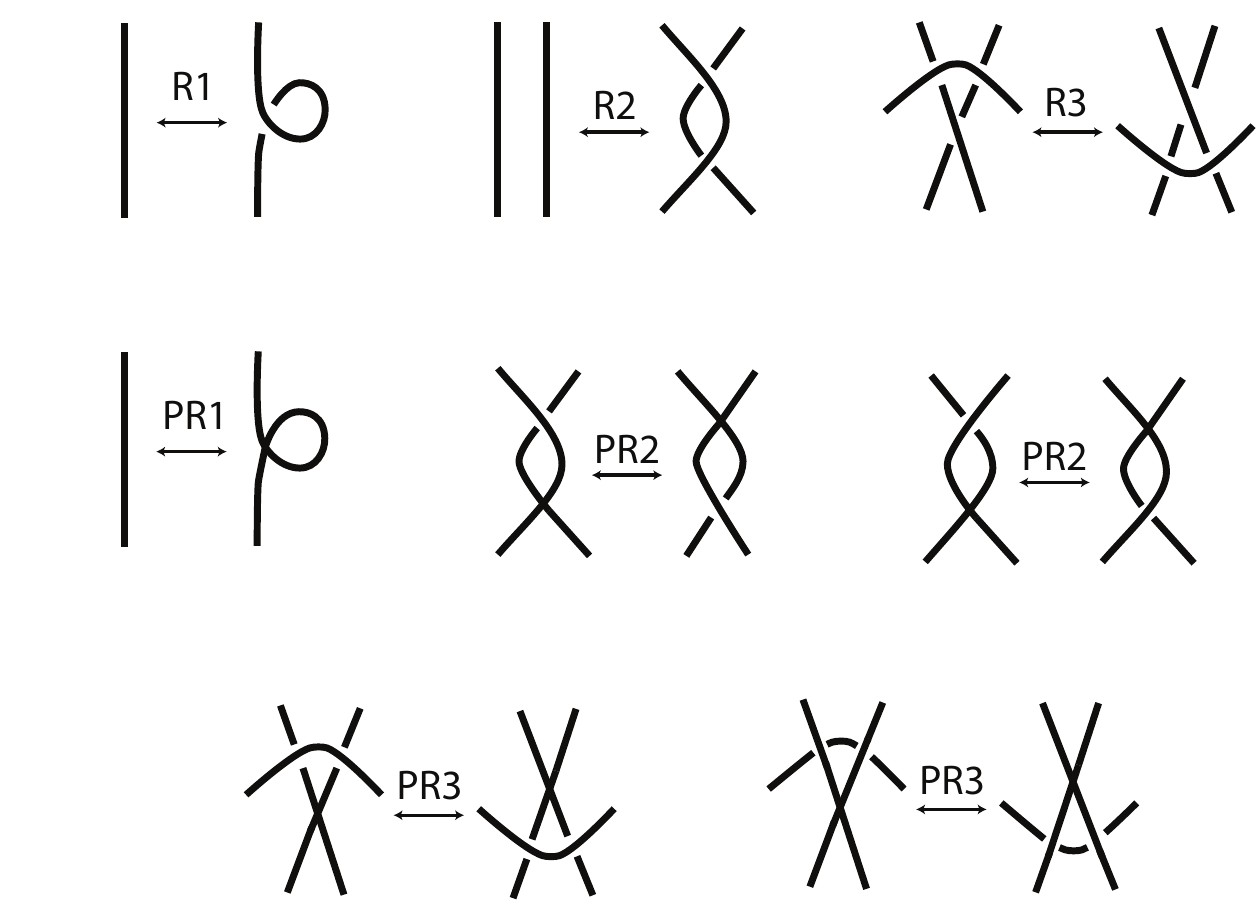}
\caption{{\bf Classical and Pseudo-Reidemeister moves}}
\label{Rmoves}
\end{center}
\end{figure}

At first glance, pseudo-Reidemeister moves two (PR2) and three (PR3), along with rigid vertex isotopy at the double-point(s), are equivalences akin to what is seen in the theory of singular knots.  However, the inclusion of the pseudo-Reidemeister one (PR1) move makes the theory of pseudoknots distinct.  

The most familiar representation of a given knot is that of a knot diagram which is a shadow of the knot decorated with crossing information at transverse intersection points. An alternate and sometimes more useful presentation is a Gauss diagram which consists of a core circle oriented counterclockwise (drawn to represent the entire curve of the oriented knot) together with the pre-image of double-points connected by chords along the circle.  The crossing information is indicated on the chord by an arrow pointing from the over-strand to the under-strand and a sign on the chord specifying whether the crossing is left or right handed.  See Figure~\ref{G-example}.

 \begin{figure}[htbp]
\begin{center}
\includegraphics[scale=.6]{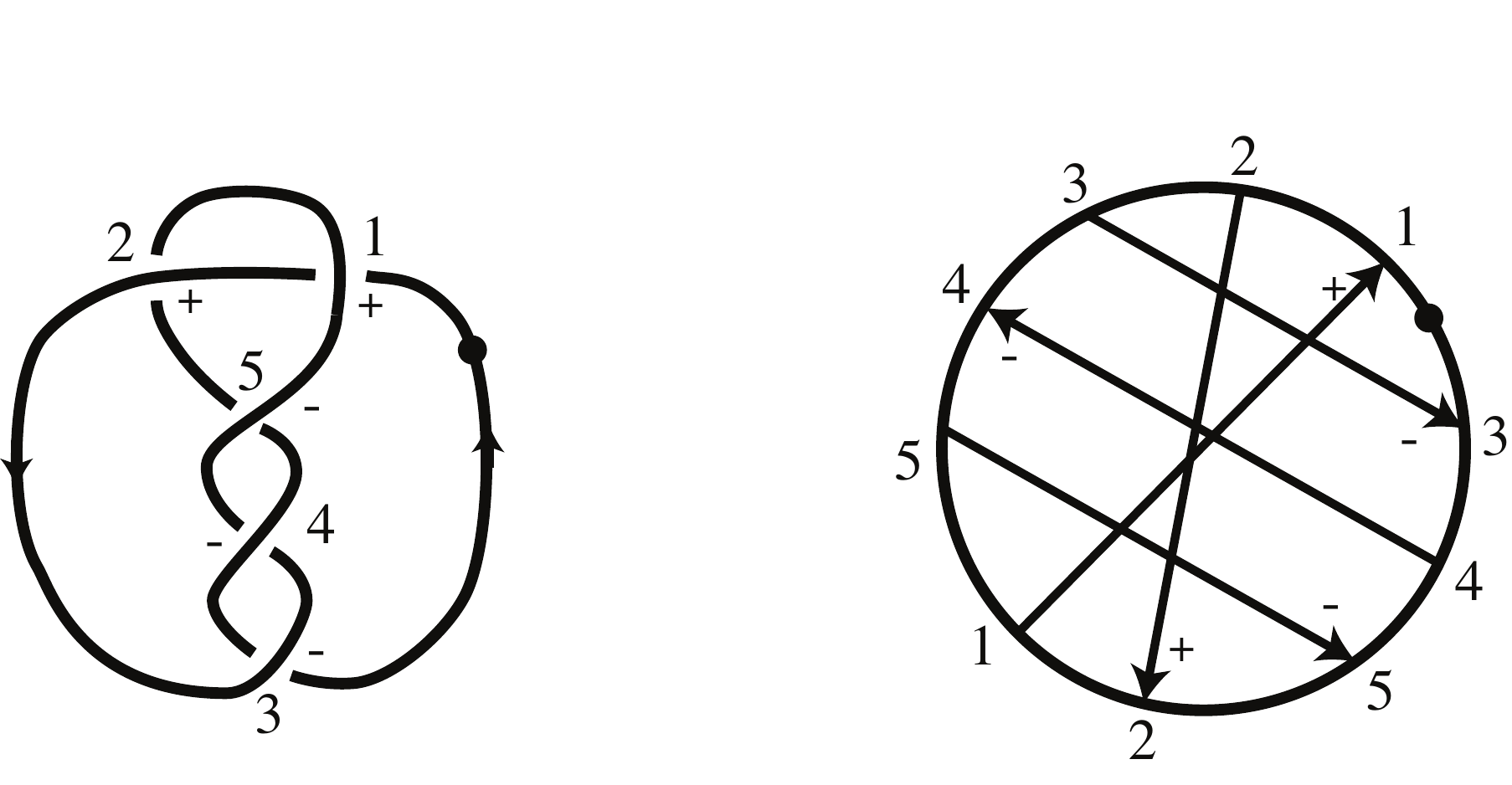}
\caption{{\bf A knot diagram and its corresponding Gauss diagram}}
\label{G-example}
\end{center}
\end{figure}

We may extend the definition of a Gauss diagram to pseudoknots as follows.  All classical crossings in a pseudoknot are represented in the Gauss diagram by a standard chord decorated with arrow and crossing sign.  A precrossing is represented by a bold or thicker chord.  We must take care, however, to indicate the proper `handedness' of the precrossing.  Figure~\ref{precrossing-ab} indicates a general precrossing and its decorated arrow within the Gauss diagram, and Figure~\ref{pseudoG-example} gives an example of a pseudoknot and its Gauss diagram.  Notice that in the Gauss diagram, the precrossing arrow points in the same direction as the classical arrow would point if the precrossing were resolved positively.

 \begin{figure}[htbp]
\begin{center}
\includegraphics[scale=0.6]{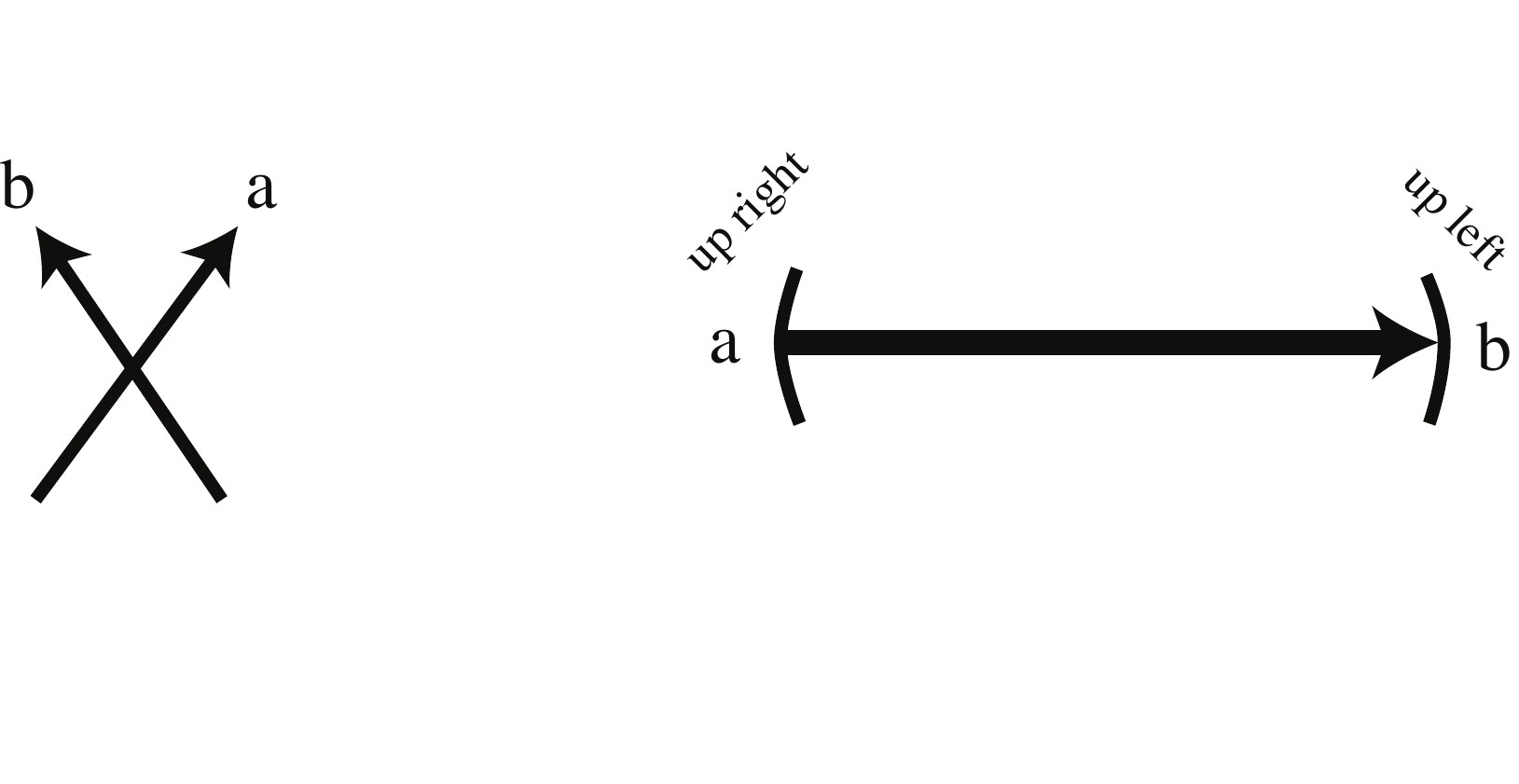}
\caption{{\bf A precrossing and a subsection of its Gauss diagram}}
\label{precrossing-ab}
\end{center}
\end{figure}

 \begin{figure}[htbp]
\begin{center}
\includegraphics[scale=.6]{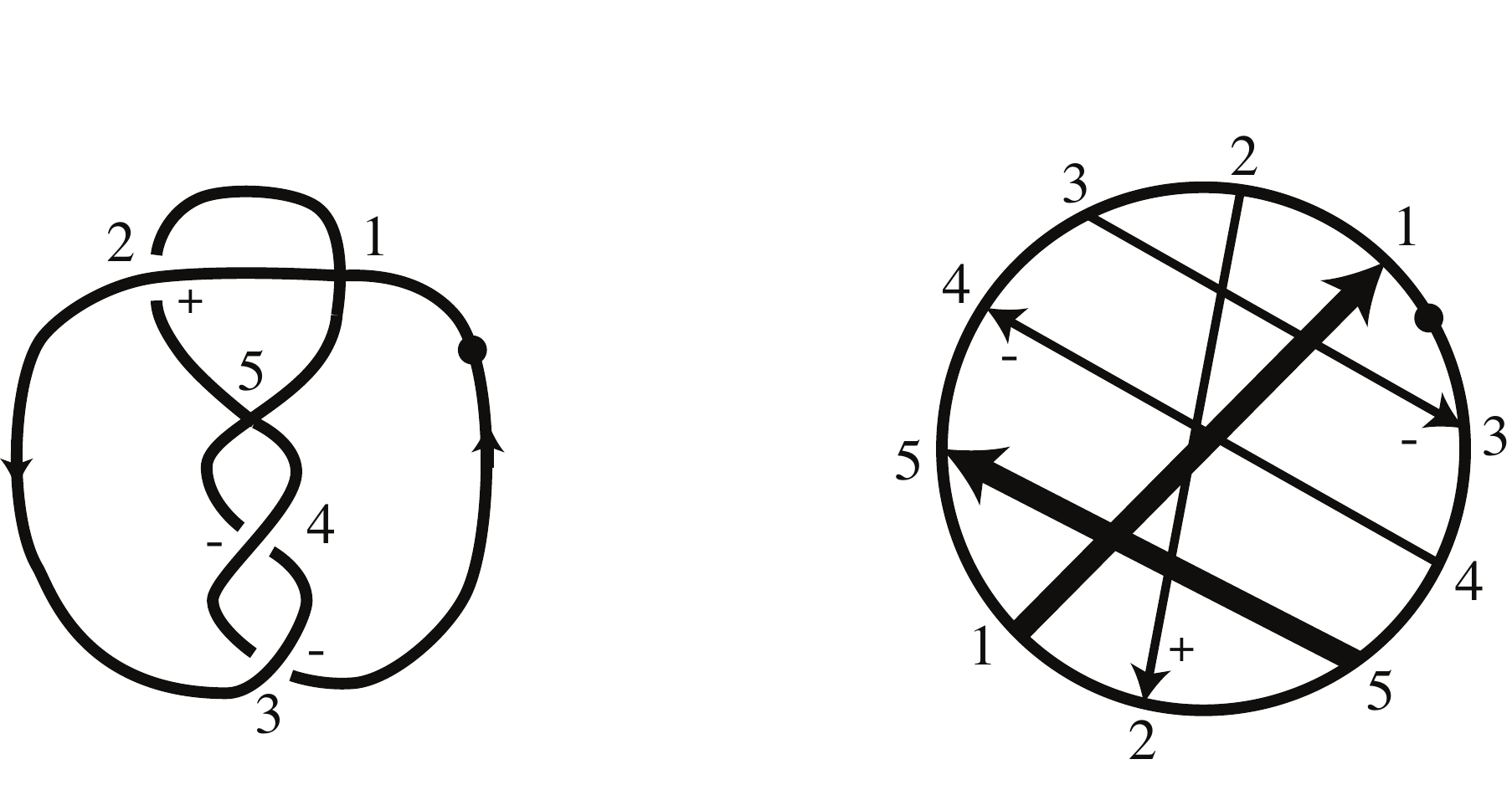}
\caption{{\bf A pseudoknot diagram and its corresponding Gauss diagram}}
\label{pseudoG-example}
\end{center}
\end{figure}

To make use of Gauss diagrams we must understand equivalence of Gauss diagrams up to classical and pseudo-Reidemeister moves.  Polyak's generating set of classical Reidemeister moves (Figure 5 in~\cite{polyak}) are represented as Gauss diagram equivalences in Figure~\ref{cR-Gauss}. Only the indicated chords appear in the solid arcs of the circle, and the portions of the diagram where the circle is dotted may include the endpoints of any collection of chords. Notice that in the R2 move, both chord arrows must point to the same section of solid arc in the Gauss diagram, which represents the arc with two undercrossings in the R2 knot diagram. Observe that within the R3 move, two chord arrows will always emanate from one of the solid arcs. This represents the arc containing two overcrossings in the R3 knot diagram. 

A subset of the complete collection of PR moves are represented as Gauss diagrams in Figure~\ref{PR-Gauss}. Four of the eight PR3 moves are shown. The other four are related to those pictured by switching all signs and arrows of the classical chords.

\begin{figure}[Htbp]
\begin{center}
\includegraphics[scale=1]{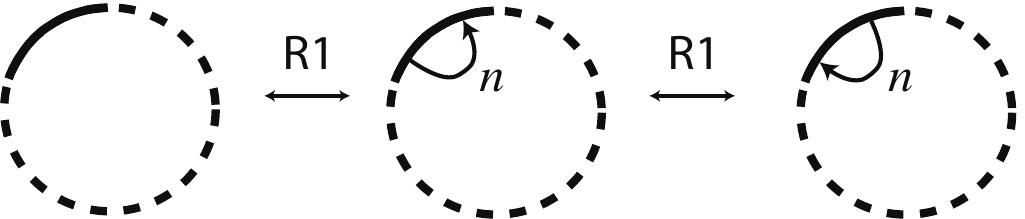}
\vspace{.2in}

\includegraphics[scale=1]{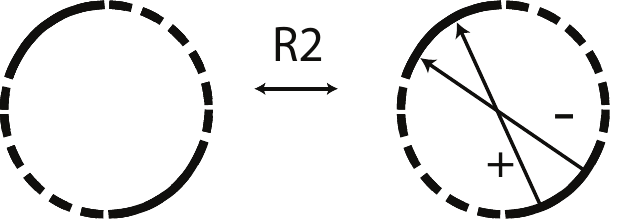}
\vspace{.2in}

\includegraphics[scale=1]{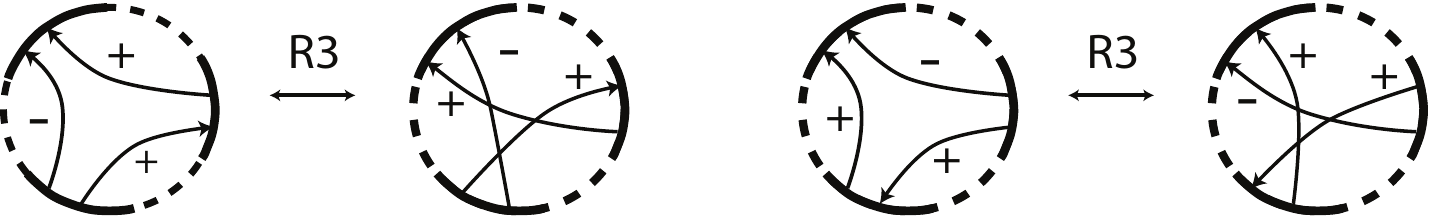}
\caption{{\bf Gauss diagrams of Polyak's classical Reidemeister moves with $n=\pm1$}}
\label{cR-Gauss}
\end{center}
\end{figure}

\begin{figure}[Htbp]
\begin{center}
\includegraphics[scale=1]{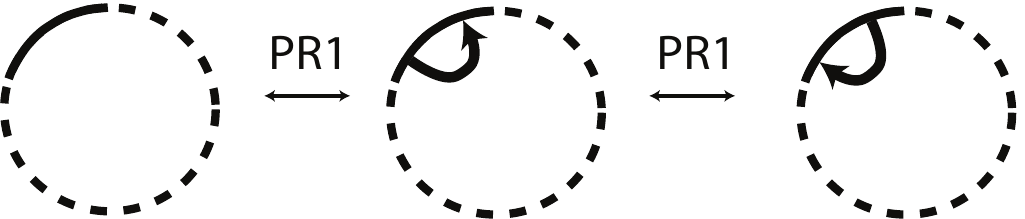}
\vspace{.2in}
\includegraphics[scale=1]{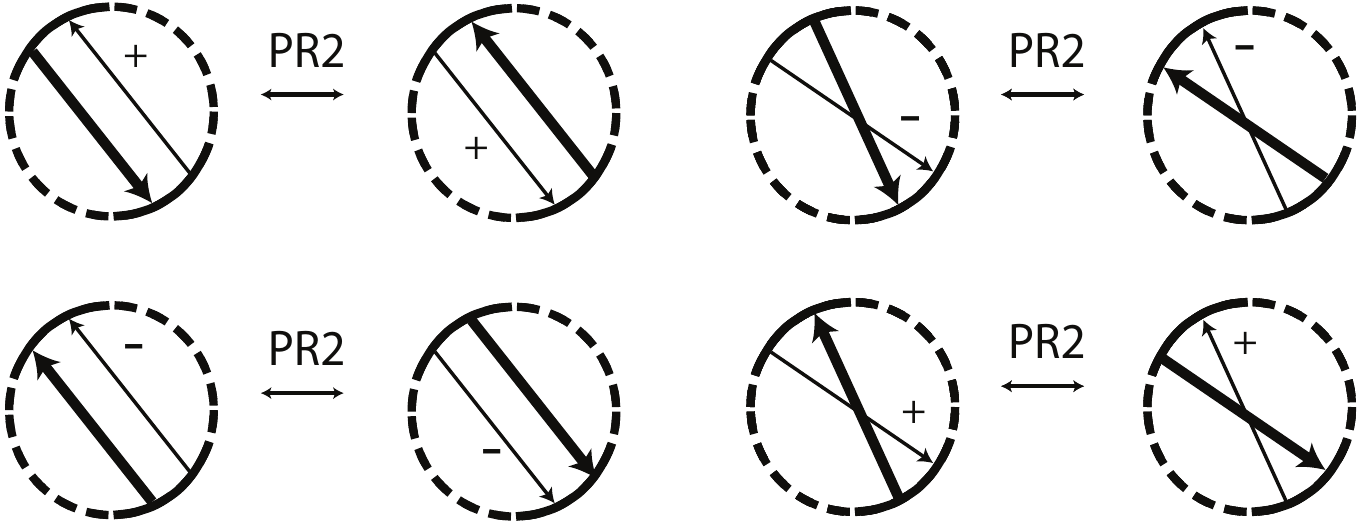}
\vspace{.2in}
\includegraphics[scale=1]{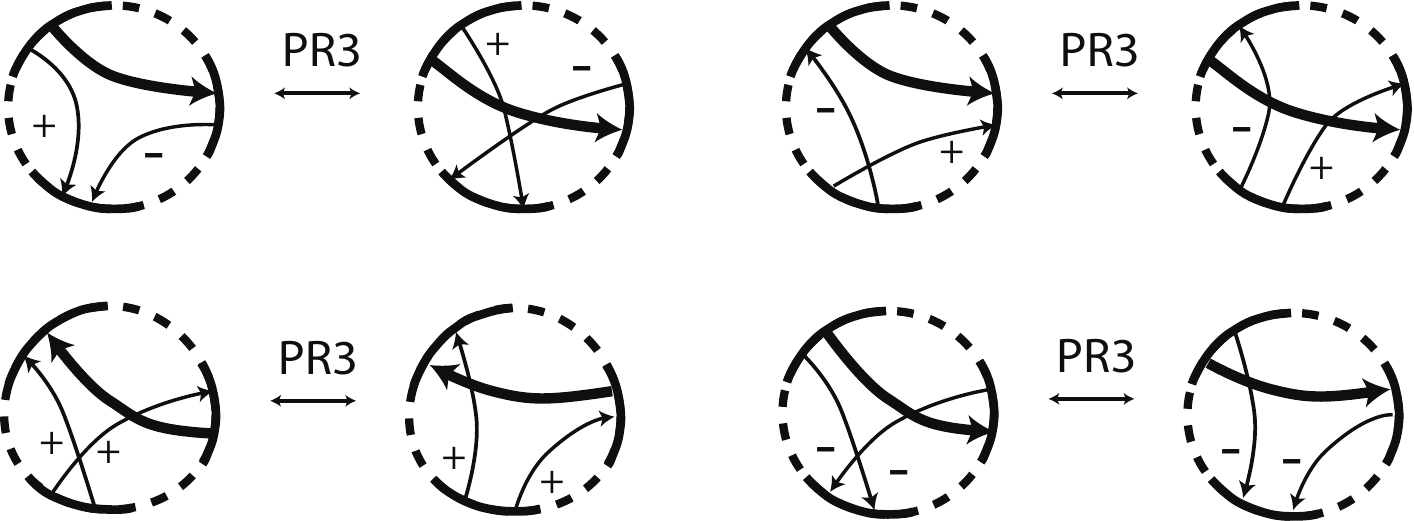}
\caption{{\bf Gauss diagrams of the pseudo-Reidemeister moves}}
\label{PR-Gauss}
\end{center}
\end{figure}

We arrive at a key observation at this point in the discussion. While every pseudoknot can be represented by this extended class of Gauss diagrams, not all Gauss diagrams are realizable as pseudoknots. In particular, there are many classical Gauss diagrams that fail to describe actual knots.  For ordinary Gauss diagrams, this observation prompted Kauffman to introduce the theory of virtual knots~\cite{vkt}. Here, we define the parallel theory of virtual pseudoknots. 

A {\em virtual pseudoknot} is an equivalence class of pseudo-Gauss diagrams, with equivalence generated by the set of Gauss-diagrammatic Reidemeister moves. Just as with ordinary virtual knots, virtual pseudoknots can be represented as pseudodiagrams where a new type of virtual crossing is allowed. (Virtual pseudodiagrams were studied extensively in~\cite{SMALL}.) We note that equivalent virtual pseudoknot diagrams are related by classical and pseudo-Reidemeister moves as well as the {\em virtual detour move}. The virtual detour move allows the replacement of any strand in the diagram that contains only virtual crossings by any other strand starting and ending at the same points that also only contains virtual crossings. Note that the virtual detour move has no effect on the Gauss diagram of the virtual pseudoknot since virtual crossings do not appear in these diagrams.

%%%Show a virtual pseudoknot diagram and a diagram related to it by a detour move. Introduce notion of homotopy for virtual pseudoknots.%%%

\section{An Isotopy Invariant}
We are now armed with the background needed to define our primary invariant, a powerful tool for distinguishing virtual pseudoknots in general and classical pseudoknots in particular.

\begin{definition}
Consider a Gauss diagram, $P$, of a virtual pseudoknot. Define a map $\mathcal{I}(P)$ as follows.
\begin{enumerate}
\item Replace with chords all arrows in $P$ that are associated to precrossings. (I.e., delete all arrowheads on precrossing arrows.) These chords will be called {\em prechords}.
\item Decorate each prechord $c$ with the integer value $i(c)$, where $i(c)$ is the sum of the signs of the classical arrows that intersect $c$.
\item Delete all classical arrows.
\item Delete any prechords $c$ that have adjacent endpoints and $i(c)=0$.
\end{enumerate}
The codomain of $\mathcal{I}$ is the set of all chord diagrams such that each chord is decorated with an integer. We refer to this set as $\mathbb{Z}\mathcal{C}$.
\end{definition}

To illustrate this definition, we include an example of a virtual pseudoknot $P$ and its corresponding decorated chord diagram $\mathcal{I}(P)$ in Figure~\ref{I_example}.

\begin{figure}[htbp]
\begin{center}
\includegraphics[scale=.35]{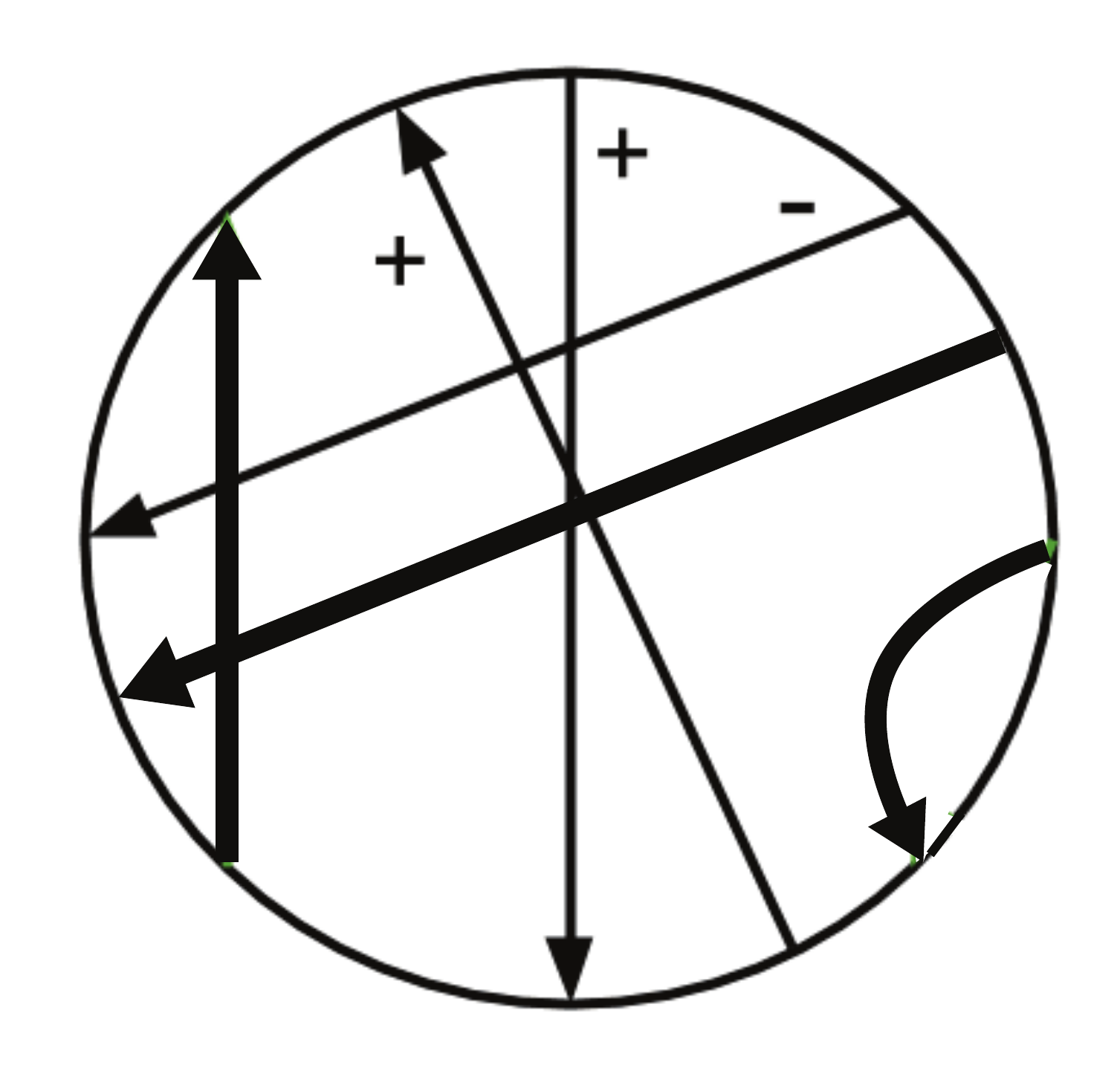}\hspace{1cm}\includegraphics[scale=.35]{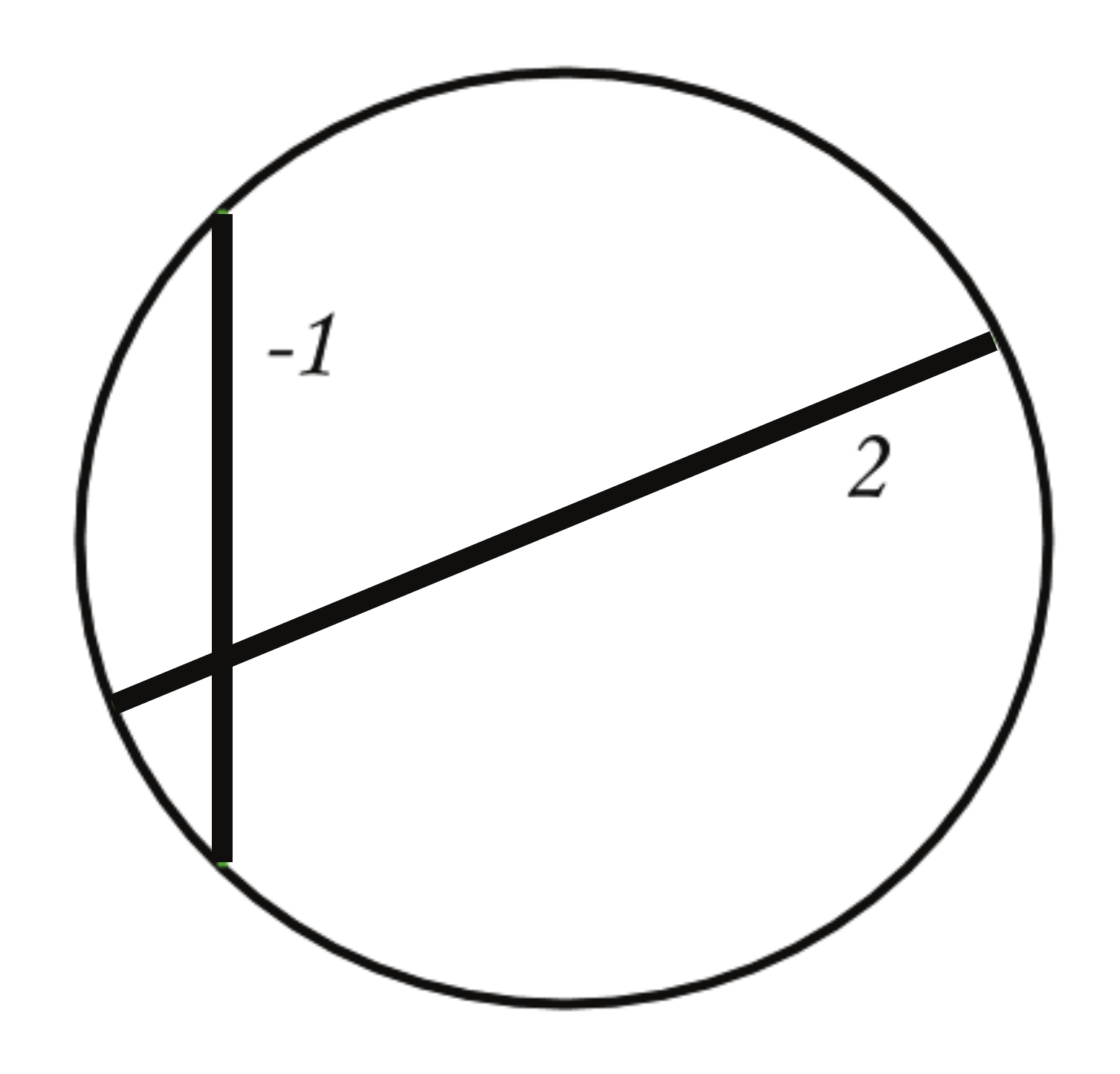}
\caption{{\bf A Gauss diagram for a virtual pseudoknot and its image under the map $\mathcal{I}$}}
\label{I_example}
\end{center}
%\vspace{-.5in}
\end{figure}

\begin{theorem} The map $\mathcal{I}$ is an invariant of virtual pseudoknots.
\end{theorem}

\begin{proof} Let us consider how each Gauss-diagrammatic Reidemeister move on a Gauss diagram $P$ affects the corresponding value of $\mathcal{I}$ in $\mathbb{Z}\mathcal{C}$. 

In terms of classical moves, R1 adds or deletes a classical arrow that doesn't intersect any precrossing arrow, so R1 doesn't change the value of $\mathcal{I}$. R2 adds or deletes two classical arrows such that, for any given precrossing arrow $a$, the two classical arrows either both intersect $a$ or both fail to intersect $a$. Since the two arrows introduced or deleted in an R2 move have opposite signs, the decoration on each prechord in $\mathcal{I}(P)$ is unchanged by this move. Classical R3 moves on Gauss diagrams do not change the signs of any classical arrows, nor do they change the intersection of a classical arrow with a precrossing arrow.

Let us turn our attention now to the effects of pseudo-Reidemeister moves on the Gauss diagram $P$. Note that the PR1 move adds or deletes a precrossing arrow with adjacent endpoints. In the computation of $\mathcal{I}$, the corresponding chord has $i(c)=0$. Hence, this chord is deleted in the process of computing $\mathcal{I}$. It follows that the values of $\mathcal{I}$ before and after a PR1 move are identical.

In a PR2 move, we notice that all intersections of arrows with the prechord are preserved, as are all signs of intersecting classical arrows. Thus, no PR2 move will alter the $i(c)$ value on this prechord. Furthermore, if the prechord involved in this move had adjacent endpoints in $\mathcal{I}$ before the PR2 move, it has adjacent endpoints in $\mathcal{I}$ after the PR2 move. In either situation, the prechord would be deleted if it has a value of $i(c)=0$ both before and after the PR2 move has been performed.

There are three types of changes that can occur to the precrossing arrow $a$ involved in a PR3 move. After the move is performed, either (a) $a$ intersects two more classical arrows that have opposite signs, (b) $a$ intersects two fewer classical arrows that have opposite signs, or (c) $a$ loses one classical arrow intersection but gains another classical arrow intersection where both arrows have the same sign. In all three cases, the value of $\mathcal{I}$ after the move is performed is the same as the value of $\mathcal{I}$ before the move. Hence, $\mathcal{I}$ is invariant under PR3.

Since we have verified invariance for R1, R2, R3, PR1, PR2, and PR3, we can conclude that $\mathcal{I}$ is indeed an invariant of virtual pseudoknots.
\end{proof}

\begin{cor} The map $\mathcal{I}$ is an invariant of (classical) pseudoknots.
\end{cor}

\section{A Homotopy Invariant}
 
At this point, we return to an interesting question that was posed for pseudoknots in~\cite{pseudoknot}. If we allow crossing changes in the classical crossings, to what extent do the precrossings obstruct the unknotting of a pseudoknot? We say that two (virtual) pseudoknots are {\em homotopic} or {\em crossing change equivalent} if any (virtual) pseudodiagrams of these two (virtual) pseudoknots are related by a sequence of PR-moves and crossing changes in the classical crossings. So our question can be reframed as follows: when are two (virtual) pseudoknots homotopic?

In terms of Gauss diagrams, the crossing change move changes both the direction of a classical arrow as well as its sign. In~\cite{pseudoknot}, it was proven that there exist nontrivial homotopy classes of pseudoknots. But how many homotopy classes exist? Using the invariant $\mathcal{I}$ as inspiration, we define a related homotopy invariant, $\mathcal{I}_h$, of virtual pseudoknots. This new invariant allows us to distinguish many pseudoknot and virtual pseudoknot homotopy classes. 

\begin{definition}
Consider a Gauss diagram, $P$, of a virtual pseudoknot. Define a map $\mathcal{I}_h(P)$ as follows.
\begin{enumerate}
\item Replace all arrows in $P$ that are associated to precrossings with prechords. 
\item Decorate each prechord $c$ with the $\mathbb{Z}_2$ value $p(c)$, where $p(c)$ is the parity of the number of classical arrows that intersect $c$.
\item Delete all classical arrows.
\item Delete any prechords $c$ that have adjacent endpoints and $p(c)=0$.
\end{enumerate}
The codomain of $\mathcal{I}_h$ is the set of all chord diagrams such that each chord is decorated with a 0 or 1. We refer to this set as $\mathbb{Z}_2\mathcal{C}$.
\end{definition}

Once again, we illustrate this definition with an example of a virtual pseudoknot $P$ and its corresponding decorated chord diagram $\mathcal{I}_h(P)$ in Figure~\ref{Ih_example}.

\begin{figure}[htbp]
\begin{center}
\includegraphics[scale=.35]{GenericEx-Bold-eps-converted-to.pdf}\hspace{1cm}\includegraphics[scale=.35]{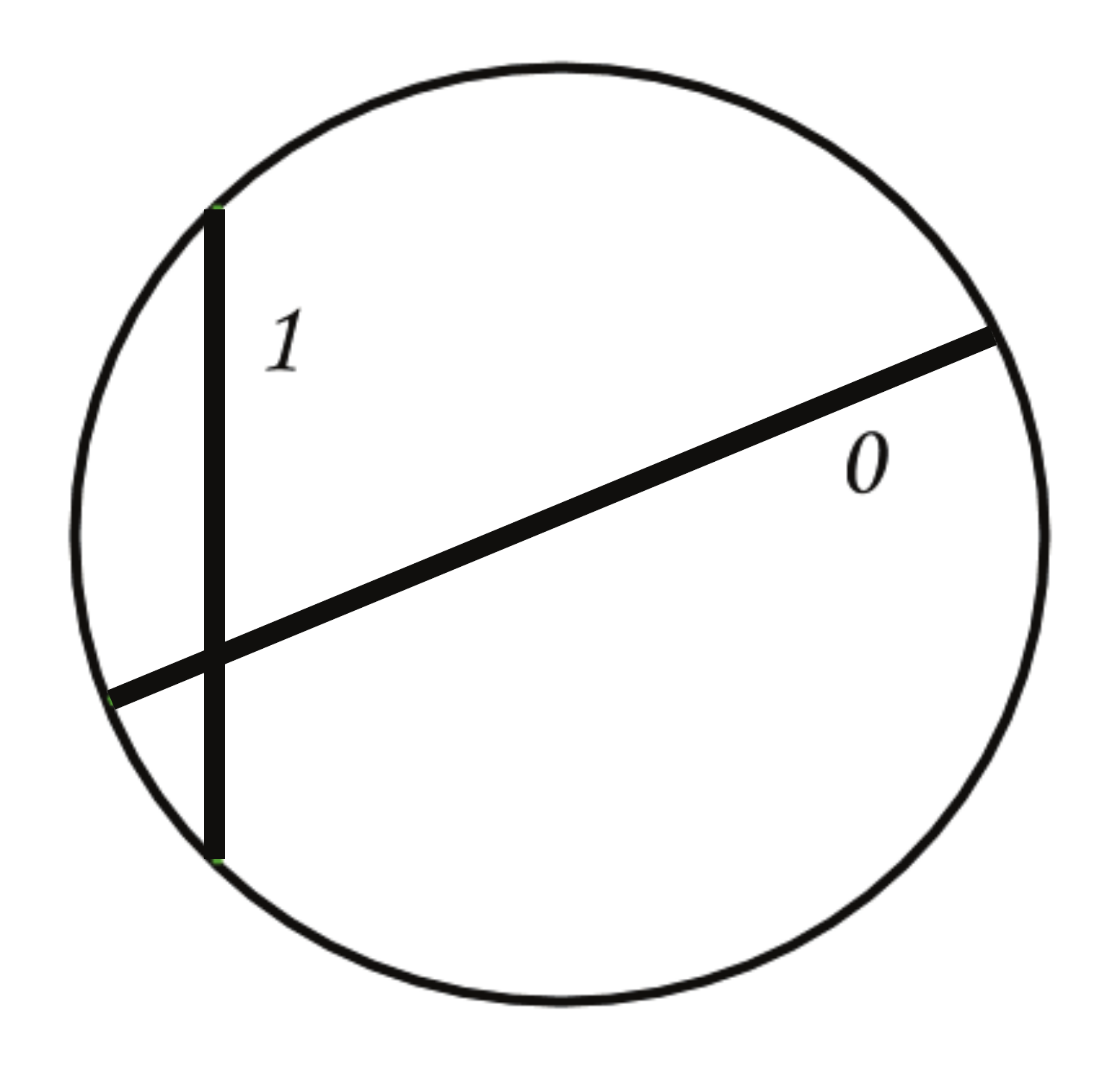}
\caption{{\bf A Gauss diagram for a virtual pseudoknot and its image under the map $\mathcal{I}_h$}}
\label{Ih_example}
\end{center}
\end{figure}

\begin{theorem} The map $\mathcal{I}_h$ is a homotopy invariant of virtual pseudoknots.
\end{theorem}

\begin{proof} First, we recall that the crossing change move changes the sign and arrow direction of a classical chord in a Gauss diagram. So let us consider the Gauss-diagrammatic Reidemeister moves that we obtain if we ignore both sign and arrow decorations on all classical chords. We claim that $\mathcal{I}_h$ is invariant under this much more flexible set of moves.

We begin by noting that no classical Gauss-diagrammatic Reidemeister moves change the parity of the number of classical chords that intersect a given precrossing arrow. In particular, R1 and R3 do not change the number of intersecting classical chords, while R2 may change the number of intersecting classical chords by 2. Since the only classical information recorded in $\mathcal{I}_h$ is the parity of intersections with precrossing arrows, classical moves do not change the value of the function.

Next, we note that all PR2 moves preserve the number of classical chords that intersect a given precrossing arrow. On the other hand, one type of PR3 move preserves the number of classical chord intersections of the precrossing arrow involved in the move, while the other type changes the intersection number of this arrow by 2. In either case, parity is preserved. Condition 4 in the definition of $\mathcal{I}_h$ guarantees invariance under PR1. Hence, all PR moves on Gauss diagrams preserve $\mathcal{I}_h$. 
\end{proof}

\begin{cor} The map $\mathcal{I}_h$ is a homotopy invariant of (classical) pseudoknots.
\end{cor}

\begin{remark} We note that the class of (virtual) pseudoknots is a quotient of the class of (virtual) singular knots, so the invariants presented in this paper give rise to invariants of (virtual) singular knots as well. 
\end{remark}

\section{Unknotting Numbers and Pseudoknot Homotopy}

For classical knots, crossing change is an unknotting operation. One
of the most difficult problems in knot theory is the computation of
unknotting numbers. We review the classical unknotting theory, then extend these ideas to homotopically trivial pseudoknots. We also provide a table of unknotting numbers for homotopically trivial pseudoknots and homotopy classes for homotopically nontrivial pseudoknots.

\subsection{Unknotting Classical Knots}

\begin{definition}
The {\it unknotting number} $u(K)$ of a knot $K$ is the minimal
number of crossing changes required to obtain the unknot from the
knot $K$, where the minimum is taken over all diagrams of $K$.
\end{definition}

There are two equivalent approaches for obtaining the
unknotting number of a knot $K$:
\begin{enumerate}
\item According to the {\it classical definition}, one
is allowed to make a planar isotopy after each crossing change and
then continue the unknotting process with the newly obtained
projection, until the unknot is obtained.
\item The {\it standard definition} requires all crossing changes
to be done simultaneously in a fixed projection.
\end{enumerate}

Unfortunately, both definitions are unsuitable for calculations,
since there are infinitely many projections of any knot. From the
well known example of the knot $10_8$ (or $5\,1\,4$ in Conway
notation), given by Y.~Nakanishi ~\cite{nakanishi} and S.~Bleiler
~\cite{bleiler}, the unknotting number is not always realizable in a minimal crossing projection. 
We recall that the rational knot $5\,1\,4$ has only one minimal projection
(Figure~\ref{NakBl}a). In the minimal projection diagram of $5\,1\,4$, unknotting requires at least three crossing changes (in the crossings denoted by circles).

\begin{figure}[htbp]
\begin{center}
\includegraphics[scale=0.5]{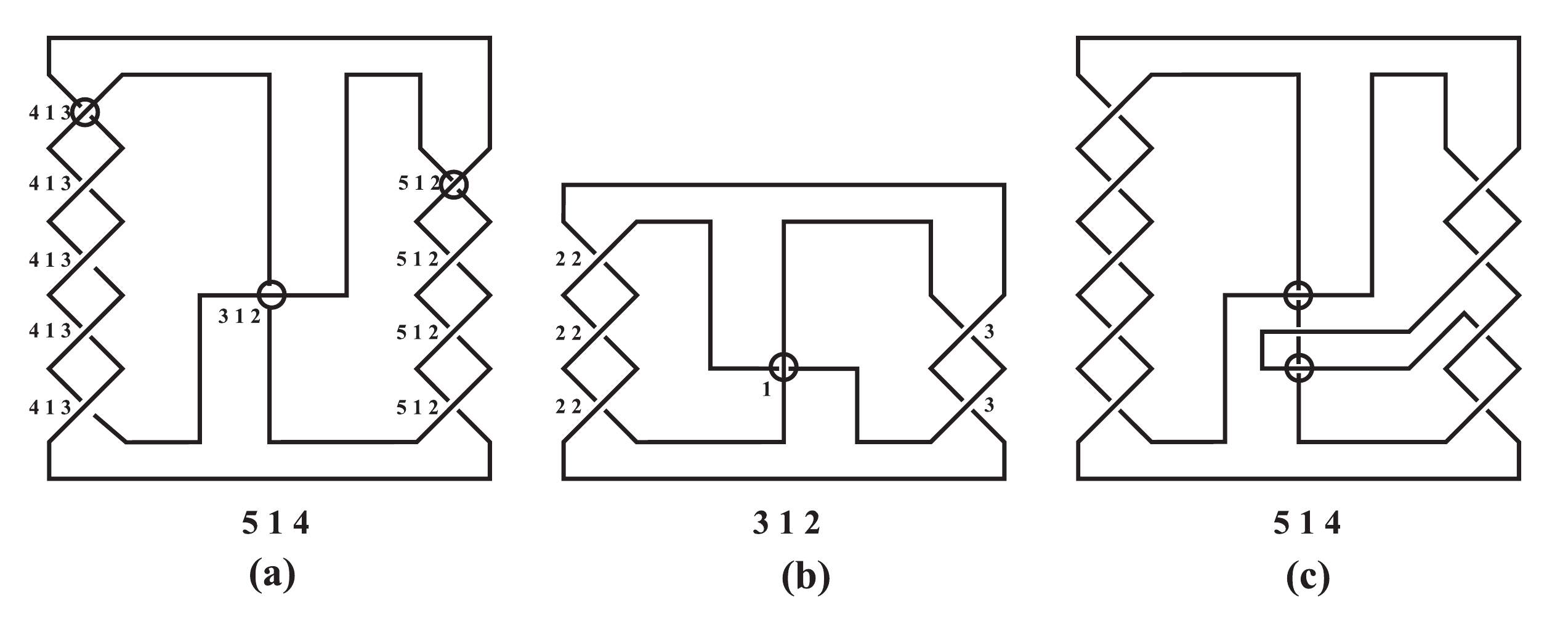}
\caption{{\bf The Nakanishi-Bleiler example: (a) the minimal
projection of the knot $5\,1\,4$ that requires at least three
crossing changes to be unknotted; (b) the minimal projection of the
knot $3\,1\,2$ with the unknotting number 1; (c) non-minimal
projection of the knot $5\,1\,4$ from which we obtain the correct
unknotting number $u(5\,1\,4)=2.$}} \label{NakBl}
\end{center}
\end{figure}

On the other hand, making a crossing change in the middle point of
the diagram (Figure~\ref{NakBl}b) followed by the reduction
$5\,-1\,4=3\,1\,2$, we obtain the minimal projection of the knot
$3\,1\,2$. This projection can be unknotted by a single crossing change. Hence, in
the case of unknotting according the classical definition, we obtain the correct unknotting number of 2 using only minimal projections. The unknotting number can also be obtained from the non-minimal
projection of the knot $5\,1\,4$ (Figure~\ref{NakBl}c) using the
standard definition.

The Nakanishi-Bleiler example motivated the definition of the
$JB$-unknotting number which is easy to compute due to the algorithmic
nature of its definition.

\begin{definition}
For a given crossing $v$ of a diagram $D$ representing knot $K$, let
$D_v$ denote the knot diagram obtained from $D$ by changing
the crossing $v$.

\begin{itemize}
    \item [a)] The {\it unknotting number} $\hat{u}(D)$ of a knot diagram $D$
is the minimum number of crossing changes required in the diagram to
obtain the unknot.
   % \item [b)] The \emph{unknotting number} of a knot $K$, denoted by $u(K)$ can be defined by $u(K)=\displaystyle{\min_D} \  \hat{u}(D)$ where the minimum is taken over all minimal crossing diagrams $D$ representing $K$.
        \item [b)] The {\it $JB$-unknotting number} $u_{JB}(D)$ of a diagram $D$
is defined recursively in the following manner:
\begin{enumerate}
    \item $u_{JB}(D)=0$ if and only if $D$ represents the unknot.
    \item $ u_{JB}(D)= 1+\displaystyle{\min_{D_v}} \ u_{JB}(D_v)$ where
     the minimum is taken over all minimal diagrams of the knot $K$
    represented by $D_v$.
\end{enumerate}
   % \item [d)]$u_M(K)= \displaystyle{\min_D} \ u(D)$ where the minimum is taken over all minimal diagrams $D$ representing $K$.
    \item [c)]The {\it $JB$-unknotting number} $u_{JB}(K)$ of a knot $K$
$u_{JB}(K)= \displaystyle{\min_D} \ u_{JB}(D)$
    where the minimum is taken over all minimal diagrams $D$ representing $K$.
  \end{itemize}

\end{definition}

J.A.~Bernhard \cite{bernhard} in 1994 and independently S.~Jablan
\cite{JS,4} in 1995 conjectured:

\begin{conjecture}
{\bf (Bernhard-Jablan Conjecture)} For every knot $K$ we have that
$u(K)=u_{JB}(K).$
\end{conjecture}

This means that we take all ($n$-crossing) minimal projections of a knot, make a single
crossing change in every crossing to obtain $n$ new knot diagrams, and then minimize all the
projections obtained. The same algorithm is applied to the first,
second, $\ldots$ $k^\text {th}$ generation of the knots obtained.
The $JB$-unknotting number is the minimum number of steps $k$ in this
recursive unknotting process. {\em Note that even if the Bernhard-Jablan Conjecture is false,
$u_{JB}(K)$ is the best known lower bound for unknotting numbers.}

\subsection{Unknotting Homotopically Trivial Pseudoknots}

The Nakanishi-Bleiler example can be directly transferred to
pseudoknots. Consider the motivating example of the pseudoknot
$(i,1,1,1,1)\,1\,(1,1,1,1)$, i.e., the pseudoknot derived from the
knot $10_8$ (or $5\,1\,4$), where one crossing in the first twist 5
is replaced by a precrossing. (See~\cite{color} for more on the Conway notation of pseudoknots.)  
Just as with our original
knot $5\,1\,4$, the fixed diagram $(i,1,1,1,1)\,1\,(1,1,1,1)$
requires at least three crossing changes to be unknotted. On the
other hand, the diagram $(i,1,1,1,1)\,-1\,(1,1,1,1)$ obtained by one
crossing change reduces to $(i,1,1)\,1\,(1,1)$, and the next
crossing change $(i,1,1)\,-1\,(1,1)$ yields the unknot.

We conclude our investigation with the following tables in which we provide JB-unknotting numbers for homotopically trivial pseudoknots. The notation used is consistent with the pseudoknot
tables that can be downloaded from the address:

\bigskip

 {\tt http://www.mi.sanu.ac.rs/vismath/pseudotabsigned1.pdf}

\bigskip

\noindent In the tables referenced above, pseudoknots are given by their ordering
numbers, Conway symbols, and signed WeRe-sets. 

In the case of classical knots, there is only one homotopy class:
the homotopy class of the unknot. Thus, unknotting numbers are always finite. Since there exist
nontrivial homotopy classes of pseudoknots, there are pseudoknots
with finite and infinite unknotting numbers in the table below. For the pseudoknots with
at most 7 crossings and finite $u_{JB}$, we provide their
$JB$-unknotting numbers in Table 1. For the remaining pseudoknots from the
tables we determined their homotopy classes.

\begin{table}[h!]
\begin{center}
\begin{tabular}{cc}
\begin{tabular}{|c|c|c|}  \hline
$K$ & Conway notation & $u_{JB}$\\ \hline \hline
$3_1.3$&$(i,1,1)$&1\\ \hline

$4_1.3$&$(i,i)\,(1,1)$&1\\ \hline

$4_1.5$&$(i,1)\,(1,1)$&1\\ \hline

$5_1.5$&$(i,1,1,1,1)$&2\\ \hline

$5_2.4$&$(i,i,i)\,(1,1)$&1\\ \hline

$5_2.6$&$(i,i,1)\,(1,1)$&1\\ \hline

$5_2.9$&$(i,1,1)\,(1,1)$&1\\ \hline

$5_2.11$&$(1,1,1)\,(i,1)$&2\\ \hline

$6_1.4$&$(i,i,i,i)\,(1,1)$&1\\ \hline

$6_1.6$&$(i,i,i,1)\,(1,1)$&1\\ \hline

$6_1.9$&$(i,i,1,1)\,(1,1)$&1\\ \hline

$6_1.12$&$(i,1,1,1)\,(1,1)$&1\\ \hline

$6_1.13$&$(1,1,1,1)\,(i,i)$&2\\ \hline

$6_1.14$&$(1,1,1,1)\,(i,1)$&2\\ \hline

$6_2.17$&$(i,1,1)\,(i)\,(1,1)$&2\\ \hline

$6_2.18$&$(i,1,1)\,(i)\,(-1,-1)$&2\\ \hline

$6_2.21$&$(i,1,1)\,(1)\,(1,1)$&1\\ \hline

$6_2.25$&$(1,1,1)\,(i)\,(1,1)$&2\\ \hline

$6_2.26$&$(1,1,1)\,(i)\,(-1,-1)$&2\\ \hline

$6_2.27$&$(1,1,1)\,(1)\,(i,i)$&2\\ \hline

$6_2.28$&$(1,1,1)\,(1)\,(i,1)$&2\\ \hline

$6_3.17$&$(i,1)\,(1)\,(i)\,(1,1)$&2\\ \hline

$6_3.18$&$(i,1)\,(1)\,(1)\,(1,1)$&1\\ \hline

$6_3.19$&$(1,1)\,(i)\,(i)\,(1,1)$&2\\ \hline

$6_3.20$&$(1,1)\,(i)\,(i)\,(-1,-1)$&2\\ \hline

$6_3.21$&$(1,1)\,(i)\,(1)\,(1,1)$&2\\ \hline

$7_1.7$&$(i,1,1,1,1,1,1)$&3\\ \hline

$7_2.4$&$(i,i,i,i,i)\,(1,1)$&1\\ \hline

$7_2.6$&$(i,i,i,i,1)\,(1,1)$&1\\ \hline

$7_2.9$&$(i,i,i,1,1)\,(1,1)$&1\\ \hline

$7_2.12$&$(i,i,1,1,1)\,(1,1)$&1\\ \hline

$7_2.15$&$(i,1,1,1,1)\,(1,1)$&1\\ \hline

$7_2.17$&$(1,1,1,1,1)\,(i,1)$&3\\ \hline

$7_3.16$&$(i,1,1,1)\,(1,1,1)$&3\\ \hline

$7_3.17$&$(1,1,1,1)\,(i,i,i)$&2\\ \hline

$7_3.18$&$(1,1,1,1)\,(i,i,1)$&2\\ \hline

$7_3.18$&$(1,1,1,1)\,(i,1,1)$&2\\ \hline

$7_4.8$&$(i,i,i)\,(1)\,(1,1,1)$&2\\ \hline

$7_4.16$&$(i,i,1)\,(1)\,(1,1,1)$&2\\ \hline

\end{tabular}
&
\begin{tabular}{|c|c|c|} \hline
  $K$ & Conway notation & $u_{JB}$\\ \hline \hline 

$7_4.22$&$(i,1,1)\,(1)\,(1,1,1)$&2\\ \hline

$7_4.23$&$(1,1,1)\,(i)\,(1,1,1)$&3\\ \hline

$7_4.24$&$(1,1,1)\,(i)\,(-1,-1,-1)$&2\\ \hline

$7_5.24$&$(i,1,1)\,(i,i)\,(1,1)$&2\\ \hline

$7_5.25$&$(i,1,1)\,(i,i)\,(-1,-1)$&2\\ \hline

$7_5.28$&$(i,1,1)\,(i,1)\,(1,1)$&2\\ \hline

$7_5.31$&$(i,1,1)\,(1,1)\,(1,1)$&2\\ \hline

$7_5.35$&$(1,1,1)\,(i,i)\,(1,1)$&2\\ \hline

$7_5.36$&$(1,1,1)\,(i,i)\,(-1,-1)$&2\\ \hline

$7_5.39$&$(1,1,1)\,(i,1)\,(1,1)$&2\\ \hline

$7_5.41$&$(1,1,1)\,(1,1)\,(i,1)$&3\\ \hline

$7_6.18$&$(i,i)\,(1,1)\,(i)\,(1,1)$&2\\ \hline

$7_6.19$&$(i,i)\,(1,1)\,(i)\,(-1,-1)$&2\\ \hline

$7_6.22$&$(i,i)\,(1,1)\,(1)\,(1,1)$&1\\ \hline

$7_6.41$&$(i,1)\,(1,1)\,(i)\,(1,1)$&2\\ \hline

$7_6.42$&$(i,1)\,(1,1)\,(i)\,(-1,-1)$&2\\ \hline

$7_6.45$&$(i,1)\,(1,1)\,(1)\,(1,1)$&1\\ \hline

$7_6.49$&$(1,1)\,(i,i)\,(i)\,(1,1)$&2\\ \hline

$7_6.50$&$(1,1)\,(i,i)\,(i)\,(-1,-1)$&2\\ \hline

$7_6.52$&$(1,1)\,(i,i)\,(1)\,(i,1)$&2\\ \hline

$7_6.53$&$(1,1)\,(i,i)\,(1)\,(1,1)$&2\\ \hline

$7_6.57$&$(1,1)\,(i,1)\,(i)\,(1,1)$&2\\ \hline

$7_6.58$&$(1,1)\,(i,1)\,(i)\,(-1,-1)$&2\\ \hline

$7_6.60$&$(1,1)\,(i,1)\,(1)\,(i,1)$&2\\ \hline

$7_6.61$&$(1,1)\,(i,1)\,(1)\,(1,1)$&2\\ \hline

$7_6.65$&$(1,1)\,(1,1)\,(i)\,(1,1)$&2\\ \hline

$7_6.66$&$(1,1)\,(1,1)\,(i)\,(-1,-1)$&2\\ \hline

$7_6.68$&$(1,1)\,(1,1)\,(1)\,(i,1)$&2\\ \hline

$7_7.19$&$(i,i)\,(1)\,(1)\,(i)\,(1,1)$&2\\ \hline

$7_7.21$&$(i,i)\,(1)\,(1)\,(1)\,(1,1)$&1\\ \hline

$7_7.37$&$(i,1)\,(1)\,(1)\,(i)\,(1,1)$&2\\ \hline

$7_7.38$&$(i,1)\,(1)\,(1)\,(1)\,(1,1)$&1\\ \hline

$7_7.39$&$(1,1)\,(i)\,(i)\,(i)\,(1,1)$&2\\ \hline

$7_7.40$&$(1,1)\,(i)\,(i)\,(i)\,(-1,-1)$&2\\ \hline

$7_7.41$&$(1,1)\,(i)\,(i)\,(1)\,(1,1)$&2\\ \hline

$7_7.42$&$(1,1)\,(i)\,(i)\,(-1)\,(-1,-1)$&2\\ \hline

$7_7.43$&$(1,1)\,(i)\,(1)\,(i)\,(1,1)$&2\\ \hline

$7_7.44$&$(1,1)\,(i)\,(1)\,(1)\,(1,1)$&2\\ \hline

$7_7.45$&$(1,1)\,(1)\,(i)\,(1)\,(1,1)$&2\\ \hline
\end{tabular} \end{tabular}
\end{center}\caption{$u_{JB}$ numbers for homotopically trivial pseudoknots}
\end{table}

\subsection{Homotopy Classes of Pseudoknots}

All pseudoknots with at most 7 crossings are divided into 53
homotopy classes. The first pseudoknot in the list that has a new homotopy class is taken to be the representative of the
class. The first 26 classes consist from more than one pseudoknot,
and the remaining 27 classes contain only one pseudoknot with at most 7 crossings.  Their
list is the following:

\bigskip

\noindent 1) $3_1.1=(i,i,i)$, $5_1.3$, $5_2.7$, $6_2.5$, $6_2.7$,
$6_3.4$, $6_3.6$, $7_1.5$, $7_2.13$, $7_3.11$, $7_3.8$, $7_4.17$,
$7_4.18$, $7_5.10$, $7_5.19$, $7_5.29$, $7_5.37$, $7_5.5$, $7_5.7$,
$7_6.43$, $7_6.46$, $7_6.54$, $7_6.62$;

\bigskip

\noindent 2) $3_1.2=(i,i,1)$, $4_1.4$, $5_1.4$, $5_2.10$, $5_2.8$,
$6_1.11$, $6_2.10$, $6_2.11$, $6_2.13$, $6_2.20$, $6_2.23$,
$6_2.24$, $6_3.10$, $6_3.13$, $6_3.14$, $6_3.16$, $6_3.8$, $7_1.6$,
$7_2.14$, $7_2.16$, $7_3.12$, $7_3.15$, $7_4.19$, $7_4.20$,
$7_4.21$, $7_5.13$, $7_5.14$, $7_5.17$, $7_5.20$, $7_5.30$,
$7_5.38$, $7_5.40$, $7_6.33$, $7_6.34$, $7_6.37$, $7_6.44$,
$7_6.47$, $7_6.48$, $7_6.51$, $7_6.55$, $7_6.56$, $7_6.59$,
$7_6.63$, $7_6.64$, $7_6.67$, $7_7.31$, $7_7.33$, $7_7.34$,
$7_7.35$, $7_7.36$;

\bigskip

\noindent 3) $4_1.1=(i,i)\,(i,i)$, $6_1.7$, $6_2.14$, $7_4.14$,
$7_6.6$, $7_6.8$, $7_7.20$, $7_7.5$, $7_7.7$;

\bigskip

\noindent 4) $4_1.2=(i,i)\,(i,1)$, $5_2.5$, $6_1.10$, $6_1.8$,
$6_2.15$, $6_2.16$, $6_2.19$, $6_2.22$, $6_3.15$, $7_2.11$,
$7_3.14$, $7_4.12$, $7_4.13$, $7_4.15$, $7_5.27$, $7_5.33$,
$7_5.34$, $7_6.11$, $7_6.12$, $7_6.14$, $7_6.21$, $7_6.25$,
$7_6.26$, $7_6.29$, $7_6.36$, $7_6.39$, $7_6.40$, $7_7.11$,
$7_7.14$, $7_7.15$, $7_7.18$, $7_7.24$, $7_7.25$, $7_7.27$,
$7_7.29$, $7_7.32$, $7_7.9$;

\bigskip

\noindent 5) $5_1.1=(i,i,i,i,i)$, $7_1.3$, $7_3.4$, $7_5.8$;

\bigskip

\noindent 6)  $5_1.2=(i,i,i,i,1)$, $6_2.6$, $6_3.9$, $7_1.4$,
$7_3.5$, $7_3.7$, $7_5.18$, $7_5.9$;

\bigskip

\noindent 7)  $5_2.1=(i,i,i)\,(i,i)$, $7_2.7$, $7_3.9$, $7_4.4$,
$7_5.21$, $7_6.15$;

\bigskip

\noindent 8)  $5_2.2= (i,i,i)\,(i,1)$, $6_1.5$, $7_2.8$, $7_3.13$,
$7_4.5$, $7_4.7$, $7_5.22$, $7_5.23$, $7_6.16$, $7_6.17$, $7_6.28$;

\bigskip

\noindent  9) $5_2.3= (i,i,1)\,(i,i)$, $6_2.12$, $6_3.5$, $7_2.10$,
$7_3.10$, $7_4.10$, $7_4.11$, $7_5.26$, $7_5.32$, $7_6.20$,
$7_6.38$, $7_7.17$;

\bigskip

\noindent 10)  $6_1.3= (i,i,i,1)\,(i,i)$, $7_4.6$, $7_6.7$,
$7_7.10$;

\bigskip

\noindent 11)  $6_1.2= (i,i,i,i)\,(i,1)$, $7_2.5$;

\bigskip

\noindent 12) $6_2.2= (i,i,i)\,(i)\,(i,1)$, $7_5.6$;

\bigskip

\noindent 13)  $6_2.3= (i,i,i)\,(1)\,(i,i)$, $7_3.6$;

\bigskip

\noindent 14)  $6_2.4= (i,i,1)\,(i)\,(i,i)$, $7_6.13$, $7_7.6$;

\bigskip

\noindent 15)  $6_2.8= (i,i,1) (i) (i,1)$, $6_2.9$, $7_5.16$,
$7_7.26$, $7_7.28$;

\bigskip

\noindent 16)  $6_3.2= (i,i)\,(i)\,(i)\,(i,1)$, $7_6.30$;

\bigskip

\noindent 17)  $6_3.3= (i,i)\,(i)\,(1)\,(i,i)$, $7_5.15$;

\bigskip

\noindent 18)  $6_3.7= (i,i)\,(1)\,(i)\,(i,1)$, $7_6.35$

\bigskip

\noindent 19)  $6_3.11= (i,1)\,(i)\,(i)\,(i,1)$, $6_3.12$, $7_6.31$,
$7_6.32$, $7_7.30$;

\bigskip

\noindent 20)  $7_4.10= (i,i,1)\,(i)\,(i,i,1)$, $7_4.9$, $7_7.16$;

\bigskip

\noindent 21) $7_5.9=(i,i,i)\,(1,1)\,(i,1)$, $7_5.10$, $7_6.24$;

\bigskip

\noindent 22)  $7_5.11= (i,i,1)\,(i,i)\,(i,1)$, $7_5.12$, $7_6.27$;

\bigskip

\noindent 23)  $7_6.10= (i,i)\,(i,1)\,(i)\,(i,1)$, $7_6.9$, $7_7.8$;

\bigskip

\noindent 24)  $7_6.23= (i,1)\,(i,i)\,(i)\,(i,1)$, $7_6.24$;

\bigskip

\noindent 25)  $7_7.12= (i,i)\,(1)\,(i)\,(i)\,(i,1)$, $7_7.13$;

\bigskip

\noindent 26)  $7_7.22= (i,1)\,(i)\,(i)\,(i)\,(i,1)$, $7_7.23$;

\bigskip

\noindent 27)-53)

$6_1.1=(i,i,i,i)\,(i,i)$,

$6_2.1=(i,i,i)\,(i)\,(i,i)$,

$6_3.1=(i,i)\,(i)\,(i)\,(i,i)$,

$7_1.1=(i,i,i,i,i,i,i)$,

$7_1.2=(i,i,i,i,i,i,1)$,

$7_2.1=(i,i,i,i,i)\,(i,i)$,

$7_2.2=(i,i,i,i,i)\,(i,1)$,

$7_2.3=(i,i,i,i,1)\,(i,i)$,

$7_3.1=(i,i,i,i)\,(i,i,i)$,

$7_3.2=(i,i,i,i)\,(i,i,1)$,

$7_3.3=(i,i,i,1)\,(i,i,i)$,

$7_4.1=(i,i,i)\,(i)\,(i,i,i)$,

$7_4.2=(i,i,i)\,(i)\,(i,i,1)$,

$7_4.3=(i,i,i)\,(1)\,(i,i,i)$,

$7_5.1=(i,i,i)\,(i,i)\,(i,i)$,

$7_5.2=(i,i,i)\,(i,i)\,(i,1)$,

$7_5.3=(i,i,i)\,(i,1)\,(i,i)$,

$7_5.4=(i,i,1)\,(i,i)\,(i,i)$,

$7_6.1=(i,i)\,(i,i)\,(i)\,(i,i)$,

$7_6.2=(i,i)\,(i,i)\,(i)\,(i,1)$,

$7_6.3=(i,i)\,(i,i)\,(1)\,(i,i)$,

$7_6.4=(i,i)\,(i,1)\,(i)\,(i,i)$,

$7_6.5=(i,1)\,(i,i)\,(i)\, (i,i)$,

$7_7.1=(i,i)\,(i)\,(i)\,(i)\,(i,i)$,

$7_7.2=(i,i)\,(i)\, (i)\,(i)\,(i,1)$,

$7_7.3=(i,i)\,(i)\,(i)\,(1)\,(i,i)$,

$7_7.4=(i,i)\,(i)\,(1)\,(i)\,(i,i)$,

$7_7.6=(i,i)\,(i)\,(i)\,(1)\, (i,1)$.

\end{document}